\documentclass[11pt]{amsart}
\usepackage{amsmath, amssymb}
\usepackage{amscd}
\usepackage{verbatim}

\newtheorem{theorem}{Theorem}[section]

\theoremstyle{definition}

\theoremstyle{remark}

\numberwithin{equation}{section}

\newcommand\style{\mathcal }          

\newcommand{\A}{\style{A}}
\newcommand{\B}{\style{B}}

\newcommand{\G}{{\rm{G}}}

\renewcommand{\H}{\style{H}}
\newcommand{\K}{\style K}

\newcommand\mn{{\style M}_n}

\newcommand\mpee{{\style M}_p}

\newcommand\mthreepee{{\style M}_{3p}}

\newcommand\mthree{{\style M}_3}

\newcommand\fni{{\mathbb F}_{\infty}}

\newcommand\omin{\otimes_{\rm min}}
\newcommand\omax{\otimes_{\rm max}}

\newcommand\cstar{{\rm C}^*}                              
\newcommand\cstarr{{\rm C}_{\rm r}^*}              

\begin{document}

\title[Crossed Products of C$^*$-Algebras with WEP]{Crossed Products of C$^*$-Algebras with the Weak Expectation Property}

\author[A.~Bhattacharya]{Angshuman Bhattacharya}
\address{University of Regina,
Department of Mathematics and Statistics, 
Regina, Saskatchewan S4S 0A2, Canada}
\author[D.~Farenick]{Douglas Farenick}
\address{University of Regina,
Department of Mathematics and Statistics,
Regina, Saskatchewan S4S 0A2, Canada}
\thanks{The work of the second author is
supported in part by the Natural Sciences and Engineering Research
Council (NSERC) of Canada.}





\begin{abstract}
If $\alpha$ is an amenable action of a discrete group $\G$ on a unital C$^*$-algebra $\A$, then the
crossed-product C$^*$-algebra $\A\rtimes_\alpha\G$ 
has the weak expectation property if and only if $\A$ has this property.
\end{abstract}

\maketitle

\section{Introduction}

A weak expectation 
on a unital C$^*$-subalgebra $\B\subset\B(\H)$ is a unital completely positive (ucp) 
linear map $\phi:\B(\H)\rightarrow\B''$ (the double commutant of $\B$)
such that $\phi(b)=b$ for every $b\in\B$. A unital C$^*$-algebra $\A$ has the weak expectation 
property (WEP) if $\pi(\A)$ admits a weak expectation for every faithful representation
$\pi$ of $\A$ on some Hilbert space $\H$. Equivalently, if $\A\subset\A^{**}\subset\B(\H_u)$ 
denotes the universal representation of $\A$, where $\A^{**}$ is the enveloping
von Neumann algebra of $\A$, then $\A$ has WEP if and only if there is a ucp map
$\phi:\B(\H_u)\rightarrow\A^{**}$ that fixes every element of $\A$.
The notion of weak expectation first arose in the work of C.~Lance on nuclear C$^*$-algebras \cite{lance1973}, where
it was shown that every unital nuclear C$^*$-algebra has WEP. Twenty years later E.~Kirchberg established a number of important
properties and characterisations of the weak expectation property in his penetrating study of exactness \cite{kirchberg1993}. 

A C$^*$-algebra $\A$ has the quotient weak expectation property (QWEP) if $\A$ is a quotient of a C$^*$-algebra with WEP.
The class of C$^*$-algebras with QWEP enjoys a number of permanence properties, many of which are enumerated in \cite[Proposition 4.1]{ozawa2004}
and originate with Kirchberg \cite{kirchberg1993}. For example, if $\A$ is a unital C$^*$-algebra with QWEP and if $\alpha$ is an amenable action of a discrete 
group $\G$ on $\A$, then the crossed product C$^*$-algebra $\A\rtimes_\alpha\G$ has QWEP \cite[Proposition 4.1(vi)]{ozawa2004}.

In contrast to QWEP,
the weak expectation property appears to have few permanence properties. For example, $\A\omin\B$ may fail to have 
WEP if $\A$ and $\B$ have WEP; one such example is furnished by $\A=\B=\B(\H)$ \cite{ozawa2003}.
In comparison, if $\A$ and $\B$ are nuclear, then so is $\A\omin\B$, and if $\A$ and $\B$ are exact, then so is $\A\omin\B$ \cite[\S10.1,10.2]{Brown--Ozawa-book}. 
 
The purpose of this note is to establish the following permanence result for WEP (Theorem \ref{main result}): 
\emph{if $\alpha$ is an amenable action of a discrete group $\G$
on a unital C$^*$-algebra $\A$, then $\A\rtimes_\alpha\G$ 
has the weak expectation property if and only if $\A$ does.}
In this regard, the weak expectation property is consistent with the analogous permanence results for 
nuclear and exact C$^*$-algebras \cite[Theorem 4.3.4]{Brown--Ozawa-book}.

Before turning to the proof, we note that Lance's definition of WEP requires
knowledge of all faithful representations of $\A$. It is advantageous, therefore, 
to have alternate ways to characterise the weak expectation property. We mention two such ways below.

\begin{theorem}\label{ki} {\rm (Kirchberg's Criterion \cite{kirchberg1993})} A unital C$^*$-algebra $\A$
has the weak expectation property if and only if
$\A\omin\cstar(\fni)=\A\omax\cstar(\fni)$.
\end{theorem}

The second description is useful in cases where one desires to fix a particular faithful representation of $\A$.

\begin{theorem}\label{wep3x3}{\rm (A Matrix Completion Criterion \cite{farenick--kavruk--paulsen2012})} If 
$\A$ is a unital C$^*$-subalgebra of $\B(\H)$, then the following statements are equivalent:
\begin{enumerate}
\item\label{w3-1} $\A$ has the weak expectation property;
\item\label{w3-2} if, given $p\in\mathbb N$ and
$X_1,X_2\in\mpee(\A)$, there exist strongly positive operators $A,B,C\in\mpee(\B(\H))$ such that $A+B+C=1$
and 
\[
Y\,=\,\left[ \begin{array}{ccc} A& X_1 & 0 \\ X_1^* & B & X_2\\ 0& X_2^*& C\end{array}\right]    
\]
is strongly positive in $\mathcal M_{3p}(\B(\H))$, then there also exist $\tilde A,\tilde B,\tilde C\in \mpee(\A)$ with the same property.
\end{enumerate}
\end{theorem}

By strongly positive one means a positive operator $A$ for which there is a real $\delta>0$ such that $A\geq\delta 1$.

Chapters 2 and 4 of the book of Brown and Ozawa \cite{Brown--Ozawa-book} shall form our main reference for facts concerning amenable groups,
amenable actions, and reduced crossed products.  

\section{The Main Result}

\begin{theorem}\label{main result} If $\alpha$ is an amenable action of a discrete group $\G$
on a unital C$^*$-algebra $\A$, then $\A\rtimes_\alpha\G$ 
has the weak expectation property if and only if $\A$ does.
\end{theorem}

\begin{proof} We begin with two preliminary observations that are independent of whether $\A$ has WEP or not. 

The first observation is that, 
because $\alpha$ is an amenable action of $\G$
on $\A$, the C$^*$-algebra $\A\rtimes_\alpha\G$ coincides with the
reduced crossed product C$^*$-algebra
$\A\rtimes_{\alpha,\rm {r}}\G$
 \cite[Theorem 4.3.4(1)]{Brown--Ozawa-book}.
The second observation is that if $\iota:\G\rightarrow\mbox{\rm Aut}(\B)$ denotes the trivial action of $\G$
on a unital C$^*$-algebra $\B$, then the action $\alpha\omax\iota$ of $\G$ on $\A\omax\B$ is amenable. (The 
action $\alpha\omax\iota$ of $\G$ on $\A\omax\B$ satisfies
$\alpha\omax\iota(g)[a\otimes b]=\alpha_g(a)\otimes b$ for all $g\in\G$, $a\in\A$, $b\in\B$ \cite[Remark 2.74]{Williams-book}.)

To prove this second fact, using the properties that define $\alpha$ as an amenable action \cite[pp. 124-125]{Brown--Ozawa-book},
let $\{T_i\}_{i}$ denote a net of finitely supported positive-valued functions $T_i:\G\rightarrow\mathcal Z(\A)$ (the centre of $\A$) such that
$\sum_{g\in\G}T_i(g)^2=1$ 
and
\[
\lim_i\left(\left\| \sum_{g\in\G} \left[ \alpha_g(T_i(s^{-1}g))-T_i(g)\right]^* \left[ \alpha_g(T_i(s^{-1}g))-T_i(g)\right]\right\|^2\right)\rightarrow 0 
\]
for all $s\in\G$. Define finitely supported positive-valued functions $\tilde T_i:\G\rightarrow\mathcal Z\left(\A\omax\B\right)$ by $\tilde T_i(g)=T_i(g)\omax 1_\B$. 
Then $\sum_{g\in\G}\tilde T_i(g)^2=1_{\A\omax\B}$ and the limiting equation above holds with $T_i$ replaced with $\tilde T_i$ and $\alpha$ replaced with $\alpha\omax\iota$.
Hence, the action $\alpha\omax\iota$ of $\G$ on $\A\omax\B$ is amenable.

Assume now that $\A$ has the weak expectation property. By Kirchberg's Criterion (Theorem \ref{ki}), 
$\A\omin\cstar(\fni)=\A\omax\cstar(\fni)$. Let $\iota:\G\rightarrow\mbox{\rm Aut}\left(\cstar(\fni)\right)$ denote the trivial action of $\G$
on $\cstar(\fni)$. Thus, 
$\alpha\omax\iota$ is an amenable action.
Hence,
\[
\begin{array}{rcl} (\A\rtimes_\alpha\G)\omin\cstar(\fni) &=& (\A\rtimes_{\alpha,\rm {r}}\G)\omin\cstar(\fni) \\ &&\\
&=& \left( \A\omin\cstar(\fni)\right)\rtimes_{\alpha\omax\iota,\rm{r}} \G \\ && \\
&=& \left( \A\omax\cstar(\fni)\right)\rtimes_{\alpha\omax\iota,\rm{r}} \G \\ && \\
&=& \left( \A\omax\cstar(\fni)\right)\rtimes_{\alpha\omax\iota} \G \\ && \\
&=& (\A\rtimes_\alpha\G)\omax\cstar(\fni)\,,
\end{array}
\]
where the final equality holds by \cite[Lemma 2.75]{Williams-book}. Another application of Kirchberg's Criterion implies that 
$\A\rtimes_\alpha\G$ has WEP.

Conversely, assume that $\A\rtimes_\alpha\G$ has the weak expectation property and that
$\A\rtimes_{\alpha,\rm {r}}\G$ is represented faithfully on a Hilbert space $\H$. Thus,
\[
 \A\,\subset\,\A\rtimes_{\alpha,\rm {r}}\G= \A\rtimes_\alpha\G  \,\subset\,\B(\H)
 \]
also represents $\A$ faithfully on $\H$. Let $\mathcal E:\A\rtimes_{\alpha,\rm {r}}\G\rightarrow\A$
denote the canonical 
conditional expectation of $\A\rtimes_{\alpha,\rm {r}}\G$ onto $\A$ \cite[Proposition 4.1.9]{Brown--Ozawa-book}. We now use the criterion of Theorem \ref{wep3x3} for WEP.

Suppose that $p\in\mathbb N$, 
$X_1,X_2\in\mpee(\A)$, and $A,B,C\in\mpee(\B(\H))$ are such that $A+B+C=1$ and the matrix 
\[
Y\,=\,\left[ \begin{array}{ccc} A& X_1 & 0 \\ X_1^* & B & X_2\\ 0& X_2^*& C\end{array}\right]    \in \mthreepee(\B(\H))
\]
is strongly positive. Because $\A\subset\A\rtimes_\alpha\G$ and because $\A\rtimes_\alpha\G$ has WEP, 
there are, by Theorem \ref{wep3x3},  $\tilde A, \tilde B, \tilde C\in\mpee(\A\rtimes_\alpha\G)$ such that
\[
\tilde Y\,=\,\left[ \begin{array}{ccc} \tilde A& X_1 & 0 \\ X_1^* & \tilde B & X_2\\ 0& X_2^*& \tilde C\end{array}\right]    \in \mthreepee(\A\rtimes_\alpha\G)
\]
is strongly positive and $\tilde A+\tilde B+\tilde C=1$. Because ucp maps preserve strong positivity, the matrix
\[
(\mathcal E\otimes{\rm id}_{\mthree})[\tilde Y]\,=\, \left[ \begin{array}{ccc} \mathcal E(\tilde A)& X_1 & 0 \\ X_1^* & \mathcal E(\tilde B) & X_2\\ 0& X_2^*& \mathcal E(\tilde C)\end{array}\right]    \in \mthreepee(\A)
\]
is strongly positive and the diagonal elements sum to $1\in \mthreepee(\A)$. Thus, $\A\subset\B(\H)$ satisfies the criterion of Theorem \ref{wep3x3} for WEP.
\end{proof}

\section{A Direct Proof in the Case of Amenable Groups}

The proof of Theorem \ref{main result} relies on the criteria for WEP given by Theorems \ref{ki} and \ref{wep3x3}, which seem far removed
from the defining condition of Lance and thereby making the argument of Theorem \ref{main result} somewhat indirect.
The purpose of this section is to present a more conceptual proof in the case of amenable discrete groups using
Lance's definition of WEP directly together with basic facts about amenable groups and C$^*$-algebras. 

In what follows, $\lambda$ shall denote the left regular representation of $\G$ on the Hilbert space $\ell^2(\G)$ and $e$ denotes the identity of $\G$. Two properties
that an amenable group $\G$ is well known to have are:
\begin{enumerate}
\item [{(i)}] $\A\rtimes_\alpha\G=\A\rtimes_{\alpha,\rm {r}}\G$, for every unital C$^*$-algebra $\A$, and
\item[{(ii)}] $\G$ admits a F{\o}lner net---namely a net $\{F_i\}_{i\in\Lambda}$ of finite subsets $F_i\subset\G$ such that, for every $g\in\G$,
\[
\lim_i \frac{\vert F_i\cap gF_i\vert}{\vert F_i \vert}\,=\,1\,.
\]
\end{enumerate} 
(In fact the second property above characterises amenable groups.)

\begin{theorem}\label{mr amnbl} If $\alpha$ is an action of an amenable discrete group $\G$
on a unital C$^*$-algebra $\A$, then $\A\rtimes_\alpha\G$ 
has the weak expectation property if and only if $\A$ does.
\end{theorem}

\begin{proof} 
Assume first that $\A\rtimes_\alpha\G$ has the weak expectation property.
To show that $\A$ has WEP, it is sufficient to show that if $\A$ is represented faithfully as a unital C$^*$-subalgebra of $\B(\K)$, for some Hilbert space $\K$, and if
$\pi_{u}^{\A}:\A\rightarrow\B(\H_u^\A)$ is the universal representation of $\A$, then there a ucp map $\omega:\B(\K)\rightarrow\A^{**}$ such that $\omega(a)=\pi_u^\A(a)$ for every $a\in\A$.

To this end,
let $\A\rtimes_\alpha\G\subset \B(\H_u^{\A\rtimes_\alpha\G})$ be the universal representation of $\A\rtimes_\alpha\G$. Because $\A$ is unital,
$\A$ is a unital C$^*$-subalgebra of $\A\rtimes_\alpha\G $. Hence,
\[
 \A\,\subset\,\A\rtimes_\alpha\G  \,\subset\,(\A\rtimes_\alpha\G)^{**}\,\subset\, \B(\H_u^{\A\rtimes_\alpha\G})
\]
and we therefore, on the one hand, consider $\A$ as a unital C$^*$-subalgebra of $\B(\K)$, where $\K=\H_u^{\A\rtimes_\alpha\G}$.
On the other hand,
\[
\begin{array}{rcl}
 \A\,\subset\,\A\rtimes_\alpha\G  \,=\, \A\rtimes_{\alpha,\rm{r}}\G &\subset &\B(\H_u^{\A\rtimes_\alpha\G})\omin\cstarr(\G) \\ && \\
&\subset& \B(\H_u^{\A\rtimes_\alpha\G})\,\overline\otimes\, \B\left(\ell^2(\G)\right) \\ && \\
&\subset& \B\left(\K \otimes  \ell^2(\G)\right),
\end{array}
\]
where $\overline\otimes$ denotes the von Neumann algebra tensor product, yields another faithful representation of 
$\A\rtimes_\alpha\G$---in this case, as a unital C$^*$-subalgebra of $\B\left(\K \otimes  \ell^2(\G)\right)$.
Let $(\A\rtimes_\alpha\G)''$ denote the double commutant of $\A\rtimes_\alpha\G$ in $\B\left(\K \otimes  \ell^2(\G)\right)$.

Using the vector state $\tau$ on $\B\left(\ell^2(\G)\right)$ defined by $\tau(x)=\langle x\delta_e,\delta_e\rangle$ together with the identity map
$\mbox{\rm id}_{\B(\K)}:\B(\H_u^{\A\rtimes_\alpha\G})\rightarrow\B(\H_u^{\A\rtimes_\alpha\G})$, we obtain a normal ucp map 
\[
\psi=\mbox{\rm id}_{\B(\K)}\overline\otimes\tau:\B(\K)\,\overline\otimes\,\B\left(\ell^2(\G)\right)\rightarrow\B(\K).
\] 
If $\mathcal E:\A\rtimes_{\alpha,\rm {r}}\G\rightarrow\A$ denotes 
the conditional expectation of
$\A\rtimes_{\alpha,\rm {r}}\G$ onto $\A$ whereby $\mathcal E\left(\sum_g a_g\lambda_g\right)=a_e$, then, using the identification
$\A\rtimes_\alpha\G=\A\rtimes_{\alpha,\rm {r}}\G$, the restriction of $\psi$ to $(\A\rtimes_\alpha\G)''$ is a normal extension of $\rho\circ\mathcal E$,
where $\rho:\A\rightarrow\B(\K)$ is the faithful representation of $\A\subset\B\left(\K \otimes  \ell^2(\G)\right)$ as a unital C$^*$-subalgebra of $\B(\K)$.
That is, we have the following commutative diagram:
 \[
\begin{CD}
 \A\rtimes_\alpha\G@>{\mathcal E}>> \A \\
@V{}VV           @VV{\rho}V  \\
(\A\rtimes_\alpha\G)''  @>>{\psi}>  \B(\K)\,. \\ 
\end{CD}
\]
Because $\psi$ is normal, the range of $\psi_{\vert (\A\rtimes_\alpha\G)''}$ is determined by
\[
\psi\left( (\A\rtimes_\alpha\G)''\right) \,=\,\overline{ \left( \psi(\A\rtimes_\alpha\G)\right)}^{\rm SOT}\,=\, \overline{ \left( \rho(\A)\right)}^{\rm SOT}.
\]
In other words, the range of $\psi_{\vert (\A\rtimes_\alpha\G)''}$ is the strong-closure of the C$^*$-subalgebra $\A$ of $\A\rtimes_\alpha\G$ in the
enveloping von Neumann algebra $(\A\rtimes_\alpha\G)^{**}$ of $\A\rtimes_\alpha\G$. Therefore, by \cite[Corollary 3.7.9]{Pedersen-book}, 
there is an isomorphism $\theta: \overline{ \left( \rho(\A)\right)}^{\rm SOT}\rightarrow \A^{**}$ such that $\pi_u^\A=\theta_{\vert\rho(\A)}$.

Now let $\pi_0:(\A\rtimes_\alpha\G)^{**}\rightarrow (\A\rtimes_\alpha\G)''$ be the normal epimorphism that extends the identity map of $\A\rtimes_\alpha\G$.
Because $\A\rtimes_\alpha\G$ has WEP,
there is a ucp map $\phi_0:\B(\H_u^{\A\rtimes_\alpha\G})\rightarrow\left(\A\rtimes_\alpha\G\right)^{**}$ that fixes every element of 
$\A\rtimes_\alpha\G$. Hence, if $\omega=\theta\circ\psi_{\vert (\A\rtimes_\alpha\G)''}\circ\pi_0\circ\phi_0$, then $\omega$ is a ucp map of $\B(\K)\rightarrow \A^{**}$ for which
$\omega(a)=\pi_u^\A(a)$ for every $a\in \A$. That is, $\A$ has WEP.

Conversely, 
assume that $\A$ has the weak expectation property and that $\A$ is (represented faithfully as) a unital C$^*$-subalgebra of $\B(\H)$
for some Hilbert space $\H$. Thus, we consider $\A$ and $\A\rtimes_\alpha\G$ faithfully represented via
\[
 \A\,\subset\,\A\rtimes_\alpha\G  \,=\, \A\rtimes_{\alpha,\rm{r}}\G \,\subset \,  \B\left(\H \otimes  \ell^2(\G)\right).
\]
Note that $\mathfrak u:\G\rightarrow\B(\H_u^{\A\rtimes_\alpha\G})$ whereby $\mathfrak u(g)=\pi_u^{\A\rtimes_\alpha\G}(1\otimes \lambda_g)$ is a unitary
representation of $\G$ such that $(1\otimes\lambda)\times\pi$ is the regular (covariant) representation associated with the dynamical system $(\A, \alpha,\G)$.

Let $\pi_u^{\A\rtimes_\alpha\G}:\A\rtimes_\alpha\G\rightarrow\B(\H_u^{\A\rtimes_\alpha\G})$ be the universal representation of 
$\A\rtimes_\alpha\G$ and define $\pi:\A\rightarrow\B(\H_u^{\A})$ by $\pi=\pi_u^{\A\rtimes_\alpha\G}{}_{\vert \A\rtimes_\alpha\G}$.
Because $\pi$ is a faithful representation of $\A$ and $\A$ has WEP, there is a ucp map 
\[
\phi_0:\B(\H)\rightarrow\pi(\A)'' \subset \pi_u^{\A\rtimes_\alpha\G}(\A\rtimes_\alpha\G)''
\]
such that
$\phi_0\left(\pi(a)\right)=\pi(a)$ for every $a\in\A$.

As in \cite[Proposition 4.5.1]{Brown--Ozawa-book}, if $F\subset\G$ is a finite nonempty subset and if
$p_F\in\B(\ell^2(\G))$ is the projection with range
$\mbox{Span}\{\delta_f\,:\,f\in F\}$, then $p_F\B(\ell^2(\G))p_F$ is isomorphic to the matrix algebra $\mn$ for $n=|F|$, and so 
we obtain a ucp map $\phi_F:\B(\H\otimes\ell^2(\G)\rightarrow \B(\H)\otimes \mn$ defined by 
$\phi_F(x)=(1\otimes p_F)x(1\otimes p_F)$.
Next, let $\{e_{f,h}\}_{f,h\in F}$ denote the matrix units of $\mn$ and
define an action $\beta$ of $\G$ on $\pi(\A)''$ by $\beta_g(y)=\mathfrak u(g)y\mathfrak u(g)^*$, for $y\in\pi(\A)''$. Observe that
$\pi(\A)''\rtimes_\beta\G\subset\pi_u^{\A\rtimes_\alpha\G}(\A\rtimes_\alpha\G)''$.

The linear map
$\psi_F:\pi(\A)''\otimes\mn\rightarrow\A\rtimes_\beta\G$ for which $\psi_F(y\otimes e_{f,h})=|F|^{-1}\beta_f(y)\mathfrak{u}(fh^{-1})$, for $y\in\pi(\A)''$,
is a ucp map by the proof of \cite[Lemma 4.2.3]{Brown--Ozawa-book}. Hence,
$\theta_F:=\psi_F\circ(\phi_0\otimes{\rm id}_{\mn})\circ\phi_F$ is a ucp map $\B\left(\H \otimes  \ell^2(\G)\right)\rightarrow\pi_u^{\A\rtimes_\alpha\G}(\A\rtimes_\alpha\G)''$.

Hence, if $\{F_i\}_i$ is a F{\o}lner net in $\G$ and if $\theta_i:\B\left(\H \otimes  \ell^2(\G)\right)\rightarrow\pi_u^{\A\rtimes_\alpha\G}(\A\rtimes_\alpha\G)''$
is the ucp map constructed above, for each $i$, then the net $\{\theta_i\}_i$ admits a cluster point $\theta$ relative to the point-ultraweak topology.
Now, for every $i\in\Lambda$, $a\lambda_g\in \A\rtimes_{\alpha,\rm{r}}\G$, and $\xi,\eta\in\H^{\A\rtimes_\alpha\G}$,
\[
\begin{array}{rcl}
\left| \langle \left( \theta(a\lambda_g)-\pi_u^{ \A\rtimes_\alpha\G}(a\lambda_g)\right) \xi,\eta\rangle \right|
&\leq&
\left| \langle \left( \theta(a\lambda_g)-\theta_{F_i}(a\lambda_g)\right) \xi,\eta\rangle \right| \\
&&\quad +\, \left| \langle \left( \theta_{F_i}(a\lambda_g)-\pi_u^{ \A\rtimes_\alpha\G}(a\lambda_g)\right) \xi,\eta\rangle \right| \\ && \\
&=& \left| \left(1-\displaystyle\frac{|F_i\cap gF_i|}{|F_i|}\right) \langle \pi_u^{ \A\rtimes_\alpha\G}(a\lambda_g)\xi,\eta\rangle \right|
\,.
\end{array}
\]
Because $\theta$ is a cluster point of $\{\theta_i\}_i$, we deduce that $ \theta(a\lambda_g)=\pi_u^{ \A\rtimes_\alpha\G}(a\lambda_g)$.
Hence, by continuity, $\theta:\B\left(\H \otimes  \ell^2(\G)\right)\rightarrow\pi_u^{\A\rtimes_\alpha\G}(\A\rtimes_\alpha\G)''$ is a ucp map for
that extends the identity map on $\pi_u^{\A\rtimes_\alpha\G}(\A\rtimes_\alpha\G)$, which proves that $\A\rtimes_\alpha\G$
has the weak expectation property.
\end{proof}

\section{Remarks}

The two proofs given in Theorems \ref{main result} and \ref{mr amnbl} of the implication \emph{$\A\rtimes_\alpha\G$ has WEP $\Rightarrow$ $\A$ has WEP}
depend only on the equality $\A\rtimes_\alpha\G=\A\rtimes_{\alpha,\rm {r}}\G$ rather than on the amenability of the action $\alpha$ or the group $\G$ itself.
 
The arguments to establish Theorems \ref{main result} and \ref{mr amnbl} 
depend crucially on the fact that $\A$ is a unital C$^*$-algebra, and it would be of interest to
know to what extent such results remain true for non-unital C$^*$-algebras.
 

\end{document}